\documentclass[12pt,a4paper]{amsart}
\usepackage{amscd,amsmath,amssymb,amsfonts,ifthen} 
 
\usepackage{bbold}
\usepackage{color} 
\usepackage[all]{xypic}
\usepackage{color}
\definecolor{refkey}{rgb}{1,0,0}
\definecolor{labelkey}{rgb}{0,0,1}

\thanks{This research is funded by Vietnam National Foundation for Science and Technology Development (NAFOSTED) under grand Nr. 101.01-2011.34.} 
\dedicatory{Dedicated to Professor H\`a Huy Kho\'ai on the occasion of his sixty-fifth  birthday}
\begin{document}
\parskip10pt

\theoremstyle{plain}
\newtheorem{thm}{Theorem}
\newtheorem{lem}[thm]{Lemma}
\newtheorem{cor}[thm]{Corollary}
\newtheorem{prop}[thm]{Proposition}
\newtheorem{rmk}[thm]{Remark}
\newtheorem{const}[thm]{Construction}
\newtheorem{rmks}[thm]{Remarks}
\newtheorem{ex}[thm]{Example}
\newtheorem{conj}[thm]{Conjecture}

\theoremstyle{definition}
\newtheorem{disc}[thm]{Discussion}
\newtheorem{question}[thm]{Question}
\newtheorem{nota}[thm]{Notations}
\newtheorem{chal}[thm]{Challenge}
\newtheorem{claim}[thm]{Claim}
 \newtheorem{defn}[thm]{Definition}
 \numberwithin{thm}{section}
\newcommand{\surj}{\twoheadrightarrow}
\newcommand{\inj}{\hookrightarrow}  
\newcommand{\sA}{{\mathcal A}}
\newcommand{\sB}{{\mathcal B}}
\newcommand{\sC}{{\mathcal C}}
\newcommand{\sD}{{\mathcal D}}
\newcommand{\sE}{{\mathcal E}}
\newcommand{\sF}{{\mathcal F}}
\newcommand{\sG}{{\mathcal G}}
\newcommand{\sH}{{\mathcal H}}
\newcommand{\sI}{{\mathcal I}}
\newcommand{\sJ}{{\mathcal J}}
\newcommand{\sK}{{\mathcal K}}
\newcommand{\sL}{{\mathcal L}}
\newcommand{\sM}{{\mathcal M}}
\newcommand{\sN}{{\mathcal N}}
\newcommand{\sO}{{\mathcal O}}
\newcommand{\sP}{{\mathcal P}}
\newcommand{\sQ}{{\mathcal Q}}
\newcommand{\sR}{{\mathcal R}}
\newcommand{\sS}{{\mathcal S}}
\newcommand{\sT}{{\mathcal T}}
\newcommand{\sU}{{\mathcal U}}
\newcommand{\sV}{{\mathcal V}}
\newcommand{\sW}{{\mathcal W}}
\newcommand{\sX}{{\mathcal X}}
\newcommand{\sY}{{\mathcal Y}}
\newcommand{\sZ}{{\mathcal Z}}
\newcommand{\A}{{\mathbb A}}
\newcommand{\B}{{\mathbb B}}
\newcommand{\C}{{\mathbb C}}
\newcommand{\D}{{\mathbb D}}
\newcommand{\E}{{\mathbb E}}
\newcommand{\F}{{\mathbb F}}
\newcommand{\G}{{\mathbb G}}
\renewcommand{\H}{{\mathbb H}}
\newcommand{\J}{{\mathbb J}}
\newcommand{\M}{{\mathbb M}}
\newcommand{\N}{{\mathbb N}}
\renewcommand{\P}{{\mathbb P}}
\newcommand{\Q}{{\mathbb Q}}
\newcommand{\R}{{\mathbb R}}
\newcommand{\T}{{\mathbb T}}
\newcommand{\U}{{\mathbb U}}
\newcommand{\V}{{\mathbb V}}
\newcommand{\W}{{\mathbb W}}
\newcommand{\X}{{\mathbb X}}
\newcommand{\Y}{{\mathbb Y}}
\newcommand{\Z}{{\mathbb Z}} 

\newcommand{\Diff}{\mathcal D\text{\it iff}\hspace{.2ex}}
\newcommand{\Der}{\mathcal D\text{\it er}\hspace{.2ex}}
\def\Qcoh{{\rm Qcoh\hspace{.2ex}}}
\def\Str{{\sf Str\hspace{.1pt}}}
\def\str{{\sf str\hspace{.1pt}}}
\def\iStr{\widehat{\Str}}
\def\spec{{\rm Spec\hspace{.2ex}}}
\def\Rep{\mbox{\sf Rep}}
\def\to{\longrightarrow}
 
\author{Ph\`ung H\^o Hai}
\address{Institute of Mathematics, Hanoi}
\email{phung@math.ac.vn}
\title[Gau\ss-Manin stratification and  fundamental group schemes]{Gau\ss-Manin stratification and stratified fundamental group schemes
} 
\subjclass[2010]{14F05; 14F35, 14L17}
\keywords{Stratified bundle, Gau\ss-Manin stratification, homotopy sequence}

\maketitle  
\begin{abstract}
We define the zero-th Gau\ss-Manin stratification of a stratified bundle with respect to a smooth morphism and use it to study the homotopy sequence of stratified fundamental group schemes.\\

\noindent{\sc R\'esum\'e}: On d\'efinit la stratification de Gauss-Manin d'un fibr\'e stratifi\'e
relativement à un morphism lisse et l'utilise pour étudier la suite
d'homotopie des groupes fondamentaux stratifiés.

\end{abstract}

\section*{Introduction}
Let $X/k$ be a smooth scheme, where $k$ is a field of characteristic 0. A connection on $X$ is an $\sO_X$-module $E$ equipped with an $\sO_X$-linear map $\nabla:\Der(X/k)\to \sE nd_k(E)$. Notice that $\Der(X/k)$ is a sheaf of Lie algebras and $\sE nd_k(E)$ is also a sheaf of Lie algebras with respect to the commutator. The connection $\nabla$ is said to be flat if it is a homomorphism of Lie algebras. Since $k$ has characteristic 0, a flat connection $\nabla$ induces an algebra homomorphism from the sheaf of differential operators $\Diff(X/k)$ to $\sE nd(X/k)$. Thus a flat connection on $E$ is the same as a $\Diff(X/k)$-module structure on $E$. This is no longer the case if $k$ has positive characteristic. In the latter case, the action of $\Diff(X/k)$ on $E$ is called the stratification. Briefly speaking, a stratification is a family of connections. This notion is due to A.~Grothendieck, see \cite{BO} for details.

The famous example, due to A.~Grothendieck, that in characteristic $p$ the notion of stratified sheaves and of flat connections are different, is given in terms of a Gau\ss-Manin connection. It is known that a stratified sheaf, which is coherent as $\sO_X$-module, is locally free. One can construct a Gau\ss-Manin connection from a family of elliptic curves, which is not locally free, cf. \cite[Chapter~1]{BO}. However, if we start with  stratified sheaves and define for them appropriately Gau\ss-Manin connections, then these become stratified sheaves. This is one of the main results of the underlying work.

A disadvantage of the notion of stratification in characteristic $p>0$ is that it requires infinitely many conditions. On the other hand, in this case, stratified sheaves have many of the good properties that the flat connections have in characteristic 0, for instance, a stratified sheaf which is coherent as an $\sO_X$-module is locally free and the category of stratified bundles is an abelian, tensor category. Tannaka duality applied to this category and the fiber functor at a rational point yields a pro-algebraic group scheme, called the stratified fundamental group scheme. Many interesting properties of this group scheme have been found, see e.g. in \cite{Gieseker,santos,esnault}. 

In this work we are interested in the exactness of the homotopy sequence of the stratified fundamental groups associated to a smooth proper maps. The question is still open in its general formulation. We show however the exactness of the homotopy sequence of the solvable quotients. The proof is based on the base change property of the 0-th stratified Gau\ss-Manin connection. Our method also allows one to show the K\"unneth type formula for the product of smooth schemes (one of which is proper). This fact has indeed been proved implicitly by Gieseker (by another method)\cite[Prop.~2.4]{Gieseker}.

\section{Stratified sheaves}\label{sec_str}

 In this section we extend the equivalence of the notions of stratified bundles and flat bundles, given in \cite{Gieseker} to the relative setting.

Let $X\stackrel f\to S$ be a smooth map of  schemes over $k$, where $k$ is a field of characteristic $p>0$. 

Let $\Diff(X/S)$ denote the sheaf of algebras of $\sO_S$-linear
differential operators on $\sO_X$:
$$\sE nd_{\sO_S}(\sO_X)\supset \Diff(X/S)=\bigcup_m\Diff^m,$$ where $\Diff^0=\sO_X$ ($\sO_X$ acts on itself by multiplication),  and
$\Diff^m\subset \Diff^{m+1}$ are defined inductively. An operator $D\in
\sE nd_{\sO_S}(\sO_X)$ lies in $\Diff^m$
if for any $a\in \sO_X$, the operator $D_a$,
$$D_a(t)=D(at)-aD(t), t\in \sO_X$$
lies in $\Diff^{m-1}$. Elements of $\Diff^m$ are called
differential operators of order (not exceeding) $n$. One checks that $ \Diff^1=\Der(X/S)+\sO_X$.

Let $x_1,\ldots,x_k$ be local coordinates on $X/S$, i.e., $dx_1,\ldots,dx_k$ form a basis
for $\Omega^1_{X/S}$. Then $\Diff(X/S)$, {\em as a left $\sO_X$-module}, has a local basis consisting of the following operators
\begin{equation}\label{basis} D_{\mathbf n},  \mathbf n =(n_1,\ldots,n_k), D_{\mathbf n}(x^{\mathbf m})={\mathbf m\choose\mathbf n} x^{\mathbf m-\mathbf n},\end{equation}
here $x^l=x_1^{l_1}\ldots x_k^{l_k}$. We have the following multiplication rule:
\begin{equation}\label{multip_rule}
D_{\mathbf m}D_{\mathbf n}={\mathbf m+\mathbf n\choose\mathbf n} D_{\mathbf m+\mathbf n}.
\end{equation}
See \cite[16.11.2]{EGA4}.
In particular these operators commute. Note however that they don't commute with the operator of order $0$, $t_a:b\mapsto ab$. Their commuting rule is given above: $[D,t_a]=D_a$.

For instant we have $D_{(0,\ldots, 1,\ldots,0)}=\partial_{x_i}$. Note that, unlike the characteristic 0 case, these operators do not generate $\Diff(X/S)$.

A stratified sheaf is a quasi-coherent $\sO_X$-module equipped with an action of $\Diff(X/S)$, that is, an
$\sO_S$-linear algebra homomorphism 
\begin{equation}\nabla:\Diff(X/S)\to \sE nd_{\sO_S}(E)\end{equation}
such that for $a\in \sO_X$, $\nabla(a)=t_a$ - the multiplication by $a$.
Morphism between stratified sheaves are those $\sO_X$-module maps that commute with the
actions of $\Diff(X/S)$. The category of stratified sheaves over $X/S$ is denoted by $\Str(X/S)$.

We now give an alternative description of stratified sheaves, called flat sheaves. First, recall the notion of relative Frobenius.
Let $F_S$ (resp. $F_X$) denote the (absolute) Frobenius map of $S$ (resp. $X$). 
Let $X^{(1)}$ be the pull-back of $X$ by $F_S$:
$$X^{(1)}:=X\times_{F_S}S.$$
Locally, functions on $X^{(1)}$ have the form $s\otimes a$, $a\in \sO_X$, $s\in \sO_S$, with $s\otimes ta=st^p\otimes a$.
	Since the absolute Frobenius of $X$ is compatible with $F_S$, by the universal property of $X^{(1)}$ we have a map $F:X\to X^{(1)}$, the relative Frobenius of $X$. $F$ is given by the map $\sO_{X^{(1)}}\to \sO_X$, $s\otimes a\mapsto sa^p$.

 Let $E$ be a stratified sheaves on $X/S$. The restriction of $\nabla$
to $\Der(X/S)$ is a connection on $E$, denoted by 
$\nabla^{(0)}$. The $p$-curvature of this connection is zero, since $\nabla$
is an algebra homomorphism. Hence, according to a theorem of Cartier \cite[ Thm~5.1]{Katz}, the subsheaf $E^{(1)}$ of sections killed by $\Der(X/S)$ is $\sO_{X^{(1)}}$-linear, where
the action of $\sO_{X^{(1)}}$ given as follows:
\begin{equation}\label{eq:action_e1}(s\otimes a)(e)=sa^p e, a\in \sO_X,s\in \sO_S, e\in E.\end{equation}
Further the natural map
$$F^*(E^{(1)})=\sO_X\otimes_FE^{(1)}\to E,\quad a\otimes e\mapsto ae$$
is an isomorphism.

Consider $E^{(1)}$ as a sheave on $X^{(1)}$. The stratification $\nabla$ induces on it a connection, denoted by $\nabla^{(1)}$, 
defined as follows. Note that $\Der(X^{(1)}/S)\cong \sO_S\otimes_{F_S}\Der(X/S)$.
For $D=\sum_ia_iD_i\in \Der(X/S)$ let $D^{(1)}$ be a differential operator
of order at most $p$ and satisfies
\begin{equation}D^{(1)}(a^p)=D(a)^p, a\in \sO_X\label{D1}.\end{equation}
For example, one can take $D^{(1)}=\sum_ia_i^pD_{(0,\ldots,p,\ldots,0)}$. Now we define
\begin{equation}\nabla^{(1)}(s\otimes D):=s\nabla(D^{(1)}), s\in \sO_S, D\in\Der(X/S).\end{equation}
It is easy to check that $\nabla^{(1)}(s\otimes D)$ does not depend on the choice of $D^{(1)}$, since two such choices differ by an operator of order less than $p$ hence killing $E^{(1)}$.  Further
$$\nabla^{(1)}(s\otimes tD)=s\nabla(t^pD^{(1)})=st^p\nabla(D^{(1)})=\nabla^{(1)}(st^p\otimes D).$$
Thus $\nabla^{(1)}$ is well-defined.  
\begin{rmk}\label{rmk1}
The action of $\nabla^{(1)}$ is thus $\sO_S$-linear. For simplicity we shall, abusing language, write $\nabla^{(1)}(D)$ instead of 
$\nabla^{(1)}(1\otimes D)$.
The above constructions can be generalized to higher powers of $p$.
\end{rmk}

The following theorem is a relative version of Katz-Gieseker's results \cite{Gieseker}.
\begin{thm}\label{thm:k-g}  Let $(E,\nabla)$ be a stratified sheaf on $X/S$.
Then, for $i\geq 1$:
\begin{enumerate}
\item  The subsheaf $E^{(i)}$ of $E$:
$$E^{(i)}:=\left\{ e\in E\left| \nabla(D)(e)=0, \text{ forall } D, \text{ with \rm ord}(D)< p^i \text{ and } D(1)=0\right.\right\}$$
is a sheaf on $X^{(i)}$ with the action given by $(s\otimes a)e=sa^{p^i} e$.
In particular they are {\em$f^{-1}\sO_S$-subsheaves of $E$.}
\item The stratification $\nabla$ induces an integrable connection $\nabla^{(i)}$ with $p$-curvature 0 on $E^{(i)}$ as
an $\sO_{X^{(i)}}$-sheaf.
\item We have the equalities $E^{(i)}=\ker \nabla^{(i-1)}$
 and 
$F^*_{X^{(i)}/S}E^{(i)}\cong E^{(i-1)}$.
\item Conversely, for a sequence of quasi-coherent sheaves $E^{i}$ on $X^{(i)}$ together with isomorphisms $\sigma_i: F^*_{X^{(i)}/S}E^{i}\cong E^{i-1}$, there exists a
unique stratification on $E=E^{0}$ such that $E^{(i+1)}\cong E^{i+1}$.
\end{enumerate}
\end{thm}
\begin{proof}
For simplicity we shall give the proof for $i=1$. We first show that $(E^{(1)},\nabla^{(1)})$ is a flat connection on $X^{(1)}$.
We have ($a,b\in \sO_X$)
\begin{eqnarray*}D^{(1)}_{a^p}(b)&=& D^{(1)}(a^pb)-a^pD^{(1)}(b)\\
&=&D^{(1)}_b(a^p)+bD^{(1)}(a^p)-a^pD^{(1)}(b)\\
&=&D(a)^pb+ a^pD^{(1)}_b(1)-a^pD^{(1)}(b)\\
&=& D(a)^pb-a^pD^{(1)}(1)\\
&=& D(a)^pb.\end{eqnarray*}
Hence we have, for $a\in \sO_X, e\in E^{(1)}$,
\begin{eqnarray*} \nabla^{(1)}(1\otimes D)(a^pe)&=&a^p\nabla(D^{(1)})(e)+\nabla(D^{(1)}_{a^p})(e)\\
&=& a^p\nabla^{(1)}(D)(e)+D(a)^pe
\end{eqnarray*}
which, according to \eqref{eq:action_e1}, is the Leibniz rule for $\nabla^{(1)}$. Thus $\nabla^{(1)}$ is a connection. Its flatness is obvious as we have
$$[D,D']^{(1)}=[D^{(1)},D'^{(1)}],$$
and $\nabla$ is an algebra homomorphism.
\end{proof}
 The data $(E^i,\sigma_i:F^*_{X^{(i)}/S}E^{i+1}\cong E^{i})$ is called a flat sheaf on
 $X$. Theorem \ref{thm:k-g} provides a 1-1 correspondence between flat sheaves and stratified sheaves. As a corollary of this fact we have: 
 stratified sheaves form  a tensor category with respect to the tensor product over $\sO$, the unit object being the trivial stratified sheaf $(\sO_X,i)$ where $i$ is the embedding
$\Diff(X/S)\to \sE nd_{\sO_S}(\sO_X)$.

\begin{rmk}
If $(E,\nabla)$ is a stratified sheaf on $X/k$, where $k$ is a field and $E$ is coherent as $\sO_X$-module, then $E$ is indeed locally free, \cite[Note 2.17]{BO}. We shall call such a sheaf a {\em stratified bundle}. The category of stratified bundles on $X/k$ is denoted by $\str(X/k)$.
The proof above defines a stratification, also denoted by $\nabla^{(i)}$ on each $E^{(i)}$ as a sheaf on $X^{(i)}/k$. Moreover, the functor $(E,\nabla)\mapsto (E^{(i)},\nabla^{(i)})$ is inverse to the Frobenius functor $(E,\nabla)\mapsto F^{*i}(E,\nabla)$. In particular, $E^{(i)}$ are all locally free. A direct proof the latter fact is given in \cite{santos}.
\end{rmk}

\begin{defn}\label{h0}
A local section of $E$ is said to be horizontal if it is annihilated by differential operators $D$ with $D(1)=0$. 
The sheaf of horizontal sections of a stratified sheaf $(E,\nabla)$ is denoted by $E^\nabla$.
Further we define the 0-th stratified cohomology sheaf
$$H^0_{\rm str}(X/S,(E,\nabla)):=f_*(E^\nabla).$$
This is a sheaf on $S$.\end{defn}

\begin{lem}\label{h0_lexact}
Let $f:X\to S$ be a smooth map and $(E,\nabla)$ be a stratified sheaf on $X/S$. Then
\begin{enumerate}
\item The functor $H_{\rm str}^0(X/S,-)$ is a left exact functor from 
the category of stratified sheaves on $X/S$ to ${\sf Qcoh}(S)$.
\item We have the following equality of $\sO_S$- submodules in $f_*(E)$:
$$H_{\rm str}^0(X/S,(E,\nabla))=\bigcap_if_*E^{(i)}.$$
\item 
If $E$ is coherent $\sO_X$-module, then $H_{\rm str}^0(X/S,(E,\nabla))$ is a coherent
sheaf on $S$.
\end{enumerate}\end{lem}
\begin{proof}
i) We have $E^\nabla=\sH om_{\Diff(X/S)}(\sO_X,E)$. Thus the functor $(E,\nabla)\mapsto E^\nabla$ is left exact. We show that $f_*(E^\nabla)$ is quasi-coherent over $S$. This is a local property, so assume $S$ is affine, $S=\text{Spec} (A)$. We have
$$f_*(E^\nabla)(S)=\text{Hom}_{\Diff(X/S)}(\mathcal O,E).$$
Let $U=\text{Spec} B$ be an open subset in $S$ and $Y:=f^{-1}(U)$.
It is to show that
$$\text{Hom}_{\Diff(X/S)}(\mathcal O,E)\otimes_AB\cong
\text{Hom}_{\Diff(Y/U)}(\mathcal O,E|_Y).$$
We check this isomorphism locally on $X$. Thus we cover $X$ by affine subsets $X_i=\text{Spec} R_i$, then $Y\cup X_i=\text{Spec}( R_i\otimes_AB)$. Further we have $\Diff(R_i/A)\otimes_AB\cong\Diff(R_i\otimes_AB/B)$. Thus one is led to check:
$$\text{Hom}_{\Diff(R_i/A)}(R_i,E_i)\otimes_AB\cong
\text{Hom}_{\Diff(R_i/A)\otimes B}(R_i\otimes_AB,E_i\otimes_AB),$$
where $E_i$ denotes the restriction of $E$ to $X_i$. Notice that
$B$ is a localization of $A$ and a homomorphism on the right hand side is determined by its value at $1$, which has the form $m\otimes b$, $b\in B$, where $m$ is annihilated by differential operators $D$ with $D(1)=0$. Now this homomorphism corresponds to the homomorphism on the left hand side, which sends 1 to $m$, tensored with $b$.

ii) Let $(U_\alpha)_{\alpha\in I}$ be an open cover of $X$ by affine subsets. Then 
$$E^\nabla(U_\alpha)= \bigcap_i E^{(i)}(U_\alpha).$$
Since $f_*$ preserves inclusion, our
claim follows.

iii) It suffices to verify the claim for an affine map $f:X\to S$ with $S$ affine. Thus $S=\spec A$, $X=\spec B$ with $B$ \'etale over $A[x_1,\ldots,x_k]$. Further, by localization, we can assume that $\Diff(X/S)$ is generated over $B$ by the set $\{D_{\bf n}, \bf n\geq \bf 0\}$. Let $(E,\nabla)$ be a stratified sheaf on $X$. The inclusion $E^\nabla\to E$ is $A$-linear hence induces the $B$-linear map $B\otimes_AE^\nabla\to E$. We claim that this map is injective. Indeed, assume that $\sum_{\bf n} x^{\bf n}e_{\bf n}=0$ for some $e_{\bf n}\in E^\nabla$. Let $\bf n_0$ be maximal with respect to the lexicographic order among the indexes in this summation, ${\bf n}_0>\bf 0$. Then $D_{{\bf n}_0}(x^{\bf n})=\delta^{\bf n}_{{\bf n}_0}$ for all $\bf n$ appearing in the summation. Hence we have
$$(D_{{\bf n}_0})_{x^{\bf n}}(1)=[D_{{\bf n}_0}, t_{x_{\bf n}}](1)=\delta_{\bf n}^{{\bf n}_0}.$$
Thus, by the definition of $E^\nabla,$ $(D_{{\bf n}_0})_{x^{\bf n}}(e)=0$ for all ${\bf n}\neq {\bf n}_0$ and $e\in E^\nabla$. On the other hand, since $(D_{{\bf n}_0})_{x^{\bf n}_0}(1)=1$,
in the expression of $(D_{{\bf n}_0})_{x^{{\bf n}_0}}$ in the basis $(D_{\bf n})$, the coefficient of $D_{\bf 0}$ is non-zero. Consequently, we have
\begin{eqnarray*}0=D_{{\bf n}_0}\left(\sum_{\bf n} x^{\bf n}e_{\bf n}\right)&=&
\sum_{\bf n} x^{\bf n} D_{{\bf n}_0}(e_{\bf n})+
(D_{{\bf n}_0})_{x^{\bf n}}(e_n)\\
&=& (D_{{\bf n}_0})_{x^{{\bf n}_0}}(e_{{\bf n}_0})\neq 0,
\end{eqnarray*}
contradiction.  

Thus the map $B\otimes_AE^\nabla\to E$ is injective, hence $B\otimes_AE^\nabla$ is coherent as a $B$-module. We show that $E^\nabla$ is coherent as an $A$-module. Indeed, let $(e_i)$ be a generating set for $B\otimes_AE^\nabla$ over $B$. Write  $e_i=\sum_jb_{ji}\otimes g_j$ with $g_j\in E^\nabla$ we see that $(g_j)$ is also a generating set for $B\otimes_AE^\nabla$ over $B$. Since $B$ is faithfully flat over $A$, $(g_j)$ has to be a generating set for $E^\nabla$ over $A$.
\end{proof}


 \section{The  Gau\ss-Manin stratified sheaf}
Assume now that $f:X\to S$ is a smooth map of smooth, connected schemes over $k$, where $k$ is an algebraically closed field.
Further we assume that $f$ has connected fibers in the sense that $f_*\sO_X=\sO_S$. The first aim of this section is to show that $f_*$ yields a
functor from $\Str(X/k)\to \Str(S/k)$. Then we will show the base change property of this functor when $f$ is proper.

A stratified bundle on $X/k$ is referred to as an absolute one, while a stratified bundle on $X/S$ is referred to as a relative one. Given an absolute stratified bundle $(E,\nabla)$ on $X/k$
we can always consider it as a relative one: $(E,\nabla_{X/S})$, where $\nabla_{X/S}$ is the restriction of $\nabla$ to $\Diff(X/S)$. We can define an action of $\Der(S/k)$ on $E^{\nabla_{X/S}}$ as follows:
Consider the exact sequence of sheaves on $X$
$$0\to \Der(X/S)\to \Der(X/k)\to \sO_X\otimes \Der(S/k)\to 0.$$
Choose local coordinates $s_j$ on $S/k$ and their liftings $t_j$ to $X$ such that
$\partial_{t_j}$ kill $dx_i$, where $dx_i$ are as in the proof of Lemma \ref{h0_lexact}. Consequently $\partial_{t_j}$ commute with 
$\partial_{x_i}$ and more generally with all $D_{\mathbf n}$, $\mathbf n>0,$ given in \eqref{basis}. The liftings $t_j$ define a section to the above sequence, denoted by
$\tau$. For each $D_S\in \Der(S/k)$ we define its action on $E^{\nabla_{X/S}}$ as the action of $\tau(D_S)$
$$\nabla_S(D_S)(e)=\nabla(\tau(D_S))(e).$$

\begin{lem}\label{lem1} The subsheaf $E^{\nabla_{X/S}}$ is stable under the action
of $\nabla_S(D_S)$ and this action is independent of the choice of the lifting.
\end{lem}
\begin{proof}
The first claim amounts to saying that for a section $e$ of $E^{\nabla_{X/S}}$,
$$\nabla(D)\nabla_S(D_S)(e)=0,$$
for any $D\in \Diff(X/S)$, $D(1)=0$, and any $D_S\in \Der(S/k)$. Assume $D_S=\sum p_j\partial_{s_j}$,
$p_j\in\sO_S$. Then
$$\tau(D_S)=\sum_jp_j\partial_{t_j}.$$
Since $\partial_{x_i}$ and $\partial_{t_j}$ commute and since $\partial_{x_i}$ are
$\sO_S$-linear the conclusion follows at once.
The independence of the choice of the liftings is obvious as any two choices differ
by elements of $\Der(X/S)$ which act as zero on $E^{\nabla_{X/S}}$.
\end{proof}

\begin{thm}\label{thm1} Assume that $f:X\to S$ is a smooth map of smooth, connected schemes over a field $k$.
Let $(E,\nabla)$ be a stratified sheaf on $X/k$.
Then 
\begin{enumerate}
\item $f_*(E^{\nabla_{X/S}})$ is equipped with a stratification, denoted by $f_*\nabla$.
\item If $E$ is coherent (hence locally free) as an $\sO_X$-module then $f_*(E^{\nabla_{X/S}})$    is also locally free.\end{enumerate}\end{thm}
 \begin{proof} i) By Lemma \ref{lem1}, we have an action of $\Der(S/k)$ on $E^{\nabla_{X/S}}$ and hence on $f_*(E^{\nabla_{X/S}})$. Notice that the horizontal sections in $E^{\nabla_{X/S}}$ under this action of $\Der(S/k)$ are exactly those sections killed by $\Der(X/k)$, since they are already killed by $\Der(X/S)$. Thus these sections form a subsheaf of the sheaf $E^{(1)}:=\ker \nabla^{(0)}\subset E$, hence the sections in $f_*(E^{\nabla_{X/S}})$, killed by $\Der(S/k)$ form a subsheaf of $f_*(E^{(1)})$.
  
  We replace $E$ by $E^{(1)}$ in our discussion. Then the stratification $ {\nabla}$ on $E$ induces a stratification  on $E^{(1)}$ and repeating the above process we obtain an action of $\Der(S^{(1)}/k)$ on
 $(E^{(1)})^{\nabla_{X^{(1)}/S^{(1)}}}$... 
 Eventually we obtain a stratification on $f_*(E^{\nabla_{X/S}})$.
 
 ii) According to Lemma \ref{h0_lexact}, if $E$ is coherent then $f_*(E^{\nabla_{X/S}})$ is also coherent, hence locally free.
\end{proof}

Thus we have a left exact functor $\Str(X/k)\to \Str(S/k)$  called the $0$-th Gau\ss-Manin stratification functor. The right derived functors of this functor are called (higher) Gau\ss-Manin stratification functors. The theorem above shows that the $0$-th Gau\ss-Manin stratification of an $\sO_X$-coherent stratified sheaf on $X/k$ is $\sO_S$ coherent. This fact also holds for the relative setting: $X\xrightarrow f S\xrightarrow g T$, $f,g$ being smooth maps. The higher Gau\ss-Manin stratifications seem to be an interesting object to study. 

Assume now that $f$ is a smooth and proper morphism. In what follows we shall study the base change property for the $0$-th stratification and apply it to the homotopy sequence of stratified fundamental groups.

\noindent{\bf Notation}. For a $k$-rational point  $s:\spec k\to S$  we fix the following notations: \\
\indent \begin{tabular} {l}
 $f_s:X_s:=X\times_Ss\to \text{Spec} k$, the fiber of $f$ at $s$;\\
$S_{(s)}:=\text{Spec }\sO_{S,s}$, where $\sO_{S,s}$ the local ring at $s$;\\
  $f_{(s)}:X_{(s)}:= X\times_S S_{(s)}\to S_{(s)}$, the pull-back of $f$ to
$S_{(s)}$.
\end{tabular}
$$\xymatrix{
X_s\ar[r]\ar[d]_{f_s} & X_{(s)}\ar[r]\ar[d]^{f_{(s)}}& X\ar[d]^f\\
k\ar[r]_s& S_{(s)}\ar[r] & S
}$$

 Let $(E,\nabla)$ be a stratified bundle on $X/k$. Consider it as a relative stratified bundle $(E,\nabla_{X/S})$. For simplicity we shall usually omit the
 stratification $\nabla$ when no confusion may arise. 
 Denote $E_{s}$ the pull-back of $E$ to $X_{s}$ and $E_{(s)}$ the pull-back of $E$ to $X_{(s)}$.
 Then $E_{s}$ is a stratified bundle on $X_{s}/k$ and $E_{(s)}$ is a stratified bundle on $X_{(s)}/S_{(s)}$.
 In what follows we will also use freely the alternative description of $(E,\nabla_{X/S})$ as a sequence of vector bundles $E=(E^{(i)})$ given in Theorem \ref{thm:k-g}.

 \begin{lem}\label{lem:usc} Assume that $k=\bar k$. Then the function
 $$h^0_{\rm str}(X_s, E_s):=\dim_k H^0_{\rm str}(X_s, E_s)$$
 in the variable $s\in S(k)$ is upper semi-continuous.
 \end{lem}
 \begin{proof}
By definition we have
\begin{equation}\label{eq:usc1}
H^0_{\rm str}(X_s,E_s)=\bigcap_{i=0}^\infty H^0(X_s,E_s^{(i)}).
\end{equation}
Since 
$$H^0(X_s,E_s^{(0)})\supset H^0(X_s,E_s^{(1)})\supset \ldots $$ are vectors spaces of finite dimension
this sequence stabilizes: there exists $d=d(s)$ such that
\begin{equation}\label{eq:ds}
H^0(X_s,E_s^{(i)})= H^0(X_s,E_s^{(i+1)})
\end{equation} 
for all $i\geq d$. According to the upper continuity of $H^0(X_s,E_s)$ for each coherent sheaf $E$ on $X$
\cite[Thm~7.7.5]{EGA3}, there exists an open neighborhood $U=U(s)$ of $s$, such that for all $t\in U$,
$$h^0(X_s,E_s^{(d)})\geq h^0(X_t,E_t^{(d)}).$$
Hence 
\begin{equation}\label{eq9}
h_{\rm str}^0(X_s,E_s)=h^0(X_s,E^{(d)}_s)\geq 
h^0(X_t,E^{(d)}_t)\geq h_{\rm str}^0(X_t,E_t).
\end{equation}
\end{proof}

\begin{cor}\label{cor:usc}
There exists $d>0$ and an open subset $U$ of $S$ such that $h^0_{\rm str}(X_t, E_t)$ is a constant function in $t$ on $U$ and
that the equality in \eqref{eq:ds} holds for all $t$ in $U$ and all $i\geq d$.
\end{cor}
\begin{proof}
Choose $s$ such that $h^0_{\rm str}(X_s, E_s)$ attains it minimum and $U=U(s)$ as in the proof of the above lemma.
Then $h^0_{\rm str}(X_t, E_t)$ is a constant function on $U$. Let $d=d(s)$ as in the above proof. Then we have equalities in \eqref{eq9} for all $t\in U$. Hence the equality in \eqref{eq:ds} should hold for $s$ replaced by $t$ for all $t\in U$ and $i\geq d$.
\end{proof}

\begin{cor}\label{cor:base_change}
With $U$, $d$ as in the above corollary, we have
$$f_{(t)*}((E^{\nabla_{X/S}}){}_{(t)})=f_{(t)*}(E_{(t)}^{(d)}),$$
and hence base change holds
$$ H^0_{\rm str}(X_t,E_t)=  f_{(t)*}(E^{\nabla_{X/S}}{}_{(t)})\otimes_{\sO_{S,t}}k,$$
for all $t\in U$.
\end{cor}
\begin{proof} On $U$ the function $h^0(X_s,E_s^{(i)})$ is constant for all $i\geq d$, hence we have base change for all $E^{(i)}$, $i\geq d$, on $U$.

For a $t\in U$, let $A:=\sO_{S,t}$ and $m$ its maximal ideal. 
  Denote
$$M^{i+1}:=f_{(t)*}(E^{(i+1)}_{(t)})\subset f_{(t)*}(E^{(i)}_{(t)})=:M^{i}.$$
By base change, for $i\geq d$ the inclusion
$$M^{i+1}\otimes_Ak=H^0(X_t,E^{(i+1)}_t)\hookrightarrow H^0(X_t,E^{(i)}_t)=M^i\otimes_Ak$$
is an isomorphism, which implies
$$(M^i/M^{i+1})\otimes_Ak=0,$$
forcing $M^i/M^{i+1}=0$ by  Nakayama's Lemma.
\end{proof}
\begin{lem}\label{lem:fE}
When restricted to the open set $U\subset S$, defined as in Corollary \ref{cor:usc}, we have
$$f_*(E^{\nabla_{X/S}})=f_*(E^{(d)}).$$
\end{lem}
\begin{proof}
For all $t\in U$ we have, by assumption,
$$\bigcap_{i=0}^\infty f_*(E^{(i)}_{(t)})=f_*(E^{(d)}_{(t)})$$
Hence
$$ f_*(E^{(d)}_{(t)})=f_*(E^{(d+1)}_{(t)})=\ldots,$$
for all $t\in U$. Hence $f_*(E^{(d)})=f_*(E^{(d+1)})=\ldots$, on $U$.
\end{proof}

For each $s\in S$, let $d(s)$ be the least integer such that the equality in \eqref{eq:ds} is fulfilled for $i\geq d(s)$. 
\begin{lem}\label{lem:bound_ds} The function $d(s)$ is bounded on $S$.
\end{lem}
\begin{proof}
As shown above, there exists an open $U\subset S$, such that on $U$ the function $d(s)$ is bounded. We stratify $V:=S\setminus U$ as a disjoint union of finitely many locally closed subset $V_i$, such that each $V_i$ is itself smooth and connected. Then on each $V_i$ we can find an open $U_i$ such that $d(s)$ is bounded on $U_i$. Repeat this process. As at each step the dimension of the complement strictly decreases, our procedure  terminates.
\end{proof}
\begin{prop}\label{prop:bound_d} For each absolute stratified bundle $(E,\nabla)$ on $X/k$ there exists a constant $d$ such that for all $i\geq d$ we have
$$f_*(E^{\nabla_{X/S}})=f_*(E^{(i)})$$
as vector bundles on $S$, for all $i\geq d$. In particular, for $i\geq d$, $f_*(E^{(i)})$ are locally free.
\end{prop}
\begin{proof}   As in the proof of Corollary \ref{cor:base_change}  is suffices to check that $h^0_\str(X_s,E_s)$ is a constant functor in $s$ on the whole $S$. We know there is an open set $U$ on which this function is constant. Thus, let $s$ be a closed point out side $U$. Let $C$ be a smooth, irreducible curve in $S$ passing though $s$ and lies entirely in $U$ (except for the point $s$). Since we are interested only on the fiber $X_s$, we can  replace $S$ by ${\rm Spec} \mathcal O_{C,s}$ in our consideration.

Thus, let $S={\rm Spec} A$, where $A$ is a discrete valuation ring. Let $K$ be the quotient field of $A$ and $\eta:{\rm Spec} K\to S$ be the generic point of $S$.
 Recall the notation $M^{i}:= H^0(X,E^{(i)})$, considered as an $A$-module. We have
$$M^0\supset M^1\supset\ldots\supset M^\infty=f_*(E^{\nabla_{X/S}}), \quad  M^\infty=\bigcap_i M^i.$$
Since $\{\eta\}$ is an open set in $S$  and since $f_*(E^{\nabla_{X/S}})$ is coherent on $S$ (Lemma \ref{h0_lexact}), we have
$$M^\infty\otimes_AK=\bigcap_i (M^i\otimes_AK).$$
Hence there exists $d$ such that $M^\infty \otimes K=M^i\otimes K, \quad\forall i\geq d.$
This implies $(M^i/M^\infty)\otimes K=0,\quad \forall i\geq d.$
Consequently, $M^d/M^\infty$ is a torsion module over $A$. Since $A$ is a DVR and $M$ is finitely generated, $M^d/M^\infty$ is Artinian. This shows that $M^{d'}=M^\infty$ for some $d'\geq d$.
Thus, it suffices to show that $H^0(X,E^{(d)})$ has base change for $d\gg 0$, i.e.,
\begin{equation}\label{eq:basechange}
H^0(X,E^{(d)})\otimes_Ak\cong H^0(X_s,E_s^{(d)}).\end{equation}
Indeed, \eqref{eq:basechange} will show that $H^0(X_s, E_s^{(d)})$ is constant in $s$ for $d\gg 0$.

The reader is referred to \cite[Sections~II.9 and III.11]{hartshorne} for the basics of the theory of formal schemes.
Let $\widehat A$ be the completion of $A$ with respect to the maximal ideal. Let $\widehat X$ be the formal completion of $X$ along $X_s$ and $\widehat E$ be the completion of $E$. Since $E$ is coherent, one has
$$\widehat E\cong E\otimes_{\mathcal O_X}\mathcal O_{\widehat X}.$$
Thus $\widehat E$ is a stratified bundle on $\widehat X/\widehat A$ in the sense of \cite[Sect~4]{santos2}: $\widehat E$ is a module over the algebra of differential operators $\Diff(\widehat X/\widehat A)$, or alternatively, the stratification on $\widehat E$ is determined by the family of subsheaves $\widehat E^{(i)}$. Notice that 
$$\widehat E^{(i)}=E^{(i)}\otimes_{\mathcal O_{X^{(i)}}}\mathcal O_{\widehat X^{(i)}},$$
where $X^{(i)}:=X\times_{F_A^i}A$.
By Grothendieck's theorem on formal functions \cite[Thm~III.11.1]{hartshorne}, we have
$$H^0(\widehat X,\widehat E^{(i)})\cong \widehat{H^0(X,E^{(i)})},$$
here $\widehat{H^0(X,E^{(i)})}$ is the completion of $H^0(X,E^{(i)})$ with respect to the maximal ideal of $A$. Since $H^0(X,E^{(i)})$ is coherent as  an $A$ module (as $X/A$ is proper) we have
$$\widehat{H^0(X,E^{(i)})} =H^0(X,E^{(i)})\otimes_A\widehat A.$$
Hence
$$H_{\rm str}^0(\widehat X/\widehat A,\widehat E)=\bigcap_i H^0(\widehat X,\widehat E^{(i)})=\bigcap_i H^0(X,E^{(i)})\otimes_A\widehat A=
H^0(X,E^{(d)})\otimes_A\widehat A,$$
for $d\gg 0$, as the sequence $H^0(X,E^{(i)})$ stabilizes.
Thus \eqref{eq:basechange} will follow if we can establish the isomorphism:
\begin{equation}\label{eq:basechange_f}H_{\rm str}^0(\widehat X/\widehat A, \widehat E)\otimes_{\widehat A}k\stackrel\cong\to H^0_\str(X_s/k,E_s).\end{equation}

In order to establish \eqref{eq:basechange_f} we proceed as follows: consider $\widehat X$ as a formal scheme over $k$. Then $\widehat E=E\otimes_{\mathcal O_X}\mathcal O_{\widehat X}$ is a module over $\Diff(\widehat X/k)$ in the sense of \cite[Sect~1]{ogus}. Notice that there is a natural forgetful functor from $\str(\widehat X/k)$ to $\str(\widehat X/\widehat A)$. On the other hand, according to \cite[Prop.~1.5]{Gieseker} there is an equivalence of tensor categories between   $\str(\widehat X/k)$ and $\str(X_s/k)$. In one direction the functor is given by restriction to the special fiber, and its quasi-inverse is given by lifting a stratified bundle on $X_s/k$ to $\widehat X/k$.  In particular, $\widehat E$ is isomorphic to the lift of $E_s$ and there is an isomorphism $k$-vector spaces of horizontal sections:
$$H^0_{\str}(\widehat X/k,\widehat E)\stackrel\cong\to H^0_{\str}(X_s/k, E_s)$$
given by restriction.
Since $H^0_{\str}(\widehat X/k,\widehat E)$ is a sub $k$-vector space of $H^0_{\str}(\widehat X/\widehat A, \widehat E)$. The restriction map
$$H^0_{\str}(\widehat X/\widehat A,\widehat E)\otimes_{\widehat A}k\to H^0_{\str}(X_s/k,E_s),$$
  is surjective. Consequently, the map in \eqref{eq:basechange_f} is surjective, hence bijective. Hence \eqref{eq:basechange} is an isomorphism.  
\end{proof}
\begin{cor}
Base change holds for the $0$-th Gau\ss-Manin stratified bundles. That is, for all smooth morphism $g:T\to S$ of smooth $k$-schemes:
$$\xymatrix@C=6ex{X\times_ST\ar[r]^{\quad g_X}\ar[d]_{f_T} & X\ar[d]^f\\
T\ar[r]_g&S}$$
  we have
\begin{equation}\label{cor:bc}
g^*(f_*(E^{\nabla_{X/S}}))\cong f_{T*}(g_X^*E^{\nabla_{X/S}}).
\end{equation}
\end{cor} 
 \begin{proof}
By adjunction we have a homomorphism 
$$g^*(f_*(E^{\nabla_{X/S}}))\to f_{T*}(g_X^*(E^{\nabla_{X/S}}))$$
 of stratified sheaves on $T$. According to Proposition \ref{prop:bound_d} this is an isomorphism of the underlying vector bundles, hence it is
an isomorphism of stratified sheaves.
\end{proof}
 \begin{cor} The function $h^0_{\rm str}(X_s/k, E_s)$ is a constant function of $s$.
\end{cor}
%
%
%

\section{The homotopy exact sequence}
Let $X/k$ be a smooth, connected scheme.
The category $\str(X/k)$ of stratified bundles, equipped with the fiber functor at a $k$-point $x$ of $X$, is a Tannaka category. Tannaka duality applied to these data yields a pro-algebraic group scheme $\pi^{\rm str}(X,x)$, called the stratified group scheme of $X$ at $x$. Nori's construction applied to this functor yields an  ind-stratified principal bundle $\widehat X\to X$ under $\pi^{\rm str}(X,x)$. One can show that this pro-algebraic scheme is smooth, i.e. is a pro-limit of smooth algebraic groups.

 We also consider the commutative and the solvable quotients of $\pi^{\rm str}(X,x)$, denoted correspondingly by $\pi_{\rm comm}^{\rm str}(X,x)$ and $\pi_{\rm sol}^{\rm str}(X,x)$, as well as the toric quotient $\pi_{\rm tor}^{\rm str}(X,x)$. Thus, assuming that $k$ is algebraically closed, $\pi_{\rm sol}^{\rm str}(X,x)$ is the Tannaka dual of the subcategory of $\str(X/k)$ of objects which are iterated extension of rank 1 objects, and $\pi_{\rm tor}^{\rm str}(X,x)$ is the Tannaka dual to the subcategory generated by rank 1 objects.

Let now $k$ be an algebraically closed field, $X,S$ be smooth schemes over $k$.
Consider now a smooth, proper map $f:X\to S$ with connected fibers, i.e. $f_*\sO_X=\sO_S$. Fix $x\in X(k)$ and let $s=f(x)$, $X_s$ be the fiber of $f$ at $s$:
\begin{equation}\label{fiber}
\xymatrix{ X_s\ar[r]^i \ar[d]_{f_s}& X\ar[d]^f\\  k\ar[r]_s &S}
\end{equation} 
 The map $f$ yields a tensor functor $f^*:\str(S/k)\to \str(X/k)$ and the map $i$ yields a tensor functor $i^*:\str(X/k)\to \str(X_s/k)$, which are compatible with fibers functor at $x$ and at $s$. Thus we the sequence
\begin{equation}\label{str_homotop}
 \pi_{\rm str}(X_s,x)\to\pi_{\rm str}(X,x)\to \pi_{\rm str}(S,s)  \to 1.
\end{equation}
Analogously, we have sequences
\begin{equation}\label{str_comm}
 \pi_{\rm comm}^{\rm str}(X_s,x)\to\pi_{\rm comm}^{\rm str}(X,x)\to \pi_{\rm comm}^{\rm str}(S,s)  \to 1,
\end{equation}
\begin{equation}\label{str_solv}
 \pi_{\rm solv}^{\rm str}(X_s,x)\to\pi_{\rm solv}^{\rm str}(X,x)\to \pi_{\rm solv}^{\rm str}(S,s)  \to 1.
\end{equation}

\begin{prop}\label{prop:surj}
The homomorphism $\pi_{\rm str}(X,x)\to \pi_{\rm str}(S,s)$ is surjective. The  same claim holds for $\pi_{\rm comm}^{\rm str}$ and $\pi_{\rm solv}^{\rm str}$.
\end{prop}
 \begin{proof}The claim of Proposition \ref{prop:surj} is equivalent to saying that the image of $\str(S/k)$ under the pull-back functor $f^*$ is a full subcategory,  closed under taking subquotients (cf. \cite[Thm~2.11]{DM}).
Thus, it suffices to show, that for each stratified bundle $M$ on $S$, its pull-back $f^*(M)$ to $X$ has the property: for any $E\subset f^*(M)$ in $\str(X/k)$, there exists $N\subset M$ in $\str(S/k)$ such that $E\cong f^*(N)$ in $\str(X/k)$. 

By projection formula and by assumption that $f_*(\sO_X)=\sO_S$ we have $f_*f^*M^{(i)}=M^{(i)}$.
The restriction of $f^*(M)$ to each fiber $X_s$ is trivial in $\str(X_s/k)$, hence, this holds for the restriction of $E$ to $X_s$. In particular all the bundles $E_s^{(i)}$ are trivial on $X_s$. This holds for all $s\in S$, so that, by means of Grauert's theorem $f_*(E^{(i)})$ are locally free on $S$ and the isomorphism $\sigma_i:F^*(E^{(i+1)})\cong E^{(i)}$ yields the corresponding isomorphism for $f_*(E^{(i)})$, making it a stratified bundle on $S$, a subbundle of $M=(f_*(f^*(M^{(i)}))$.
\end{proof}

Next, we want to prove the exactness of \eqref{str_comm}.  
First we will need a lemma. Let $L\xrightarrow q G\xrightarrow p A$ be a sequence of
homomorphism of affine group scheme over a field $k$. It induces a
sequence of functors
\begin{equation}\label{a1}
\Rep(A)\xrightarrow{p^*}\Rep(G)\xrightarrow{q^*}\Rep(L)
\end{equation}
where $\Rep$ denotes the category of finite dimensional
representations over $k$.
 
\begin{lem}
  Assume that $G$ is commutative, $q:L\to G$ is a closed immersion and $p:G\to A$ is
faithfully flat (thus in particular surjective). Then the sequence $L\xrightarrow q G\xrightarrow
p A$ is exact at $G$ if and only if the following conditions are fulfilled:
 For an object $V\in \Rep(G)$, $q^*(V)$ in $\Rep(L)$ is trivial if and only if $V\cong p^*U$
for some $U\in \Rep(A)$.
 \end{lem}
\begin{proof} We show the implication ``$\Longleftarrow$", the other implication is obvious.
Since $G$ is commutative, $L$ is normal in $G$, let $B:=G/L$. Thus we have an exact sequence $1\to L\to G\to B\to 1$. Then we know that $\Rep(B)$ is a full tensor subcategory of $\Rep(G)$, closed under taking subquotients. Moreover, by means of \cite[App.~A1]{EHS}, objects $\Rep(B)$ can be characterized by the condition given in Lemma.
This means $\Rep(B)$ and $\Rep(A)$ coincide as subcategories in $\Rep(G)$, whence $B=A$.
\end{proof}

Let $\sC_S$ be the image of $\str(S/k)$ in $\str(X/k)$, under $f^*$. Proposition \ref{prop:surj} says that $\sC_S$ is a full subcategory of $\str(X/k)$, closed under
 taking subquotients. We say that a stratified bundle in $\str(X/k)$ is {\em essentially} in  $\sC_S$ if it is isomorphic to an object of $\sC_S$.
 For each object $E$ in $\str(X/k)$ there exists a maximal subobject $E_S$ which is essentially in $\sC_S$. This is because our category $\str(X/k)$ is artinian.
  
For each $E\in \str(X/k)$, let $\langle E\rangle^\otimes$ be the full tensor subcategory of $\str(X/k)$, generated by $E$.

 \begin{lem}\label{lem:restriction}
Let $E$ be a stratified bundle on $X/k$ and $E_S$ be the maximal subobject of $E$, which is essentially in $\sC_S$. Then 
\begin{equation}\label{ES}E_S=f^*(f_*(E^{\nabla_{X/S}})).
\end{equation}
  Further, for any point $s\in S(k)$, $(E_S)_s$ is the maximal trivial subobject of $E_s$ in $\str(X_s/k)$.
\end{lem}
\begin{proof} Assume that $E_S=f^*(N)$, $N\in \str(S/k)$. Then $N=f_*((E_S)^{\nabla_{X/S}})$ hence is contained in $f_*(E^{\nabla_{X/S}})$ as a subobject in $\str(S/k)$. Since $f^*(f_*(E^{\nabla_{X/S}}))$ is in $\sC_{S}$, it is contained in $E_S=f^*(N)$. Thus $N=f_*(E^{\nabla_{X/S}})$. Thus \eqref{ES} is proved.

We prove the second claim. Let $E_s{}^\nabla$ be the maximal trivial object of $E_s$ as objects in $\str(S_s/k)$. The natural map $(E_S)_s\to E_s{}^\nabla\subset E_s$ is a map of stratified bundles. Thus it suffices to show that this is an isomorphism of vector bundles. 

In the notations of \eqref{fiber} we have, by means of \eqref{ES},
$$(E_S)_s=i^* f^*f_*(E^{\nabla_{X/S}}),\quad 
E_s{}^\nabla=f_s^*f_{s*}((i^*E)^{\nabla_{X_s/k}}).$$
One is led to showing that
$$i^* f^*f_*(E^{\nabla_{X/S}})=f_s^*f_{s*}((i^*E)^{\nabla_{X_s/k}}).$$
According to Proposition \ref{prop:bound_d}, on  this equality $f_*(E^{\nabla_{X/S}} )$ can be replaced by $f_*E^{(d)}$ and $f_{s*}((i^*E)^{\nabla_{X_s/k}})$ can be replaced by $f_{s*}(i^*E^{(d)})$ for some $d>0$. Thus, it is to show:
$$i^*f^*f_*E^{(d)}= f_s^*f_{s*}(i^*E^{(d)}).$$
This equation can be easily obtained by projection formula:
\begin{eqnarray*} i^*f^*f_*E^{(d)}&= f_s^* s^*f_*(E^{(d)})\\
&= f_s^*f_{s*}i^*(E^{(d)}).
\end{eqnarray*}
\end{proof}
 
 \begin{thm}\label{thm_solv}
The sequence \eqref{str_comm} and \eqref{str_solv} are exact.
\end{thm} 
\begin{proof}
The exactness of \eqref{str_comm} follows directly from the above two lemmas. We prove the exactness of \eqref{str_solv}.

According to \cite[App.~A1]{EHS}, and Lemma \ref{lem:restriction} above it suffices to show that any object $W$ in $\str_{\rm solv}(X_s/k)$, which is a quotient object of $E_s$ for some $E$ in $\str_{\rm solv}(X/k)$, is a subobject of some $E_s'$, where $E'$ is another object in $\str_{\rm solv}(X/k)$.

 We first consider the case when $W$ is a line bundle. The Tannaka group $G(W)$ of $W$ is either finite or is isomorphic to $\mathbb G_m$, depending on whether $L$ is torsion (i.e. a tensor power of it is trivial) of not. Since the homotopy exact sequence of \'etale fundamental groups is exact, we can assume that $G(W)$ is $\mathbb G_m$. By assumption we have the following maps:
 $$\xymatrix{ G(E_s)\ar@{^(->}[r] \ar@{->>}[d]&G(E)\\ G(W)}$$
 By going to some \'etale covering can assume that $G(E)$ and $G(E_s)$ are both connected. Hence we have the following diagram with exact sequences (which are the canonical factorization of a connected solvable group scheme as an extension of a torus by a unipotent group)
 $$\xymatrix{
1\ar[r] & G(E_s)_u\ar[r]\ar[d] & G(E_s)\ar[r]\ar[d]& G(E_s)_t\ar[r]\ar[d]& 1\\
 1\ar[r] & G(E)_u\ar[r] & G(E)\ar[r]& G(E)_t\ar[r]& 1}
$$ 
Since $G(E)$ is connected, $G(E)_u$ consists of all unipotent elements of $G(E)$, and the same holds for $G(E_s)_u\subset G(E_s)$. Hence the injectivity of the middle map implies the injectivity of the other maps.
 
 Thus any representation of $G(E_s)_t$ is a subquotient of the restriction of a representation of a $G(E)_t$. Since $G(W)$ is a quotient of $G(E_s)$, we conclude that
 $W$ is a subquotient of the restriction of some object in $\str_{\rm tor}(X/k)$. In fact, since 
 object of $\str_{\rm tor}(X/k)$ are direct sums of line bundles, $W$ is the restriction of some line bundle in $\str(X/k)$.

In the general case, assume $W$ has rank $r$. Then $\bigwedge^r(W)$ is a line bundle.
If $W$ is a quotient of some $E_s$ then so is $\bigwedge^r(W)$. But then $\bigwedge^r(W)$ is isomorphic to some $V_s$. We see that $W^*\cong W\otimes \bigwedge^r(W)^*$ is also a quotient of some restriction, hence $W$ itself is a subobject of some restriction.
\end{proof}

The exactness of the sequence of the full fundamental group schemes \eqref{str_homotop} is still a question. The method presented here however allows us to handle the case of a product.
In this case we have K\"unneth product formula for the stratified fundamental groups. This result has been proved in an implicit from by Gieseker \cite[Thm. 2.4]{Gieseker}.
\begin{thm}\label{kunneth}
Let $X=Y\times_k S$, where $Y$ and $S$ are smooth, connected schemes over $k=\bar k$, and $Y$ is proper. Then the stratified fundamental group scheme of $X$ is isomorphic to the direct product
of those of $Y$ and $S$. More precisely, let $y\in Y$, $s\in S$ be $k$-points, and $x=(y,s)\in X$.
Then we have natural isomorphism
$$\pi_{\rm str}(X,x)\cong \pi_{\rm str}(Y,y)\times\pi_{\rm str}(S,s).$$
\end{thm}
\begin{proof}
Let $f:X\to S$ denote the projection to $S$ and let $g:X\to Y$ denote the projection to $Y$. Further the point $s\in S$ yields a closed embedding $i:Y=Y\times \{s\}\to Y\times S=X$.  Consider the sequence as in \eqref{str_homotop}:
$$1\to  \pi_{\rm str}(Y,y)\xrightarrow{i_*}\pi_{\rm str}(X,x)\xrightarrow{f_*} \pi_{\rm str}(S,s)  \to 1.$$
We have $g\circ i={\rm id}_Y$. Hence $i_*$ is an inclusion. And it remains to show that the sequence is exact.

According to \cite[App.~A1]{EHS} it suffices to show that each stratified bundle $E$ on $Y/k$ is embeddable into the restriction to $Y$ (by means of $i^*$) of some stratified bundle on $X/k$. But we have $E=i^*(g^*E)$ and $g^*E$ is a stratified bundle on $X$.
\end{proof}

\section*{ Acknowledgment}
The special thank goes to Prof. H\'el\`ene Esnault for many valuable comments and suggestions that shape this work. J.P. dos Santos also has some thoughts on the homotopy sequence of stratified fundamental group scheme. I thank him for sharing his ideas.
A special thank goes to Lei Zhang for his interest in the work and very helpful discussions. In particular he has pointed out a gap in the previous version of the work.
Finally I would like to thank the anonymous referee for the careful reading and helpful suggestions.


\end{document}